\documentclass[12pt,a4paper,reqno,sumlimits]{amsart}
\usepackage[margin=2.8cm]{geometry}
\usepackage[numbers]{natbib}
\usepackage{amsmath}
\usepackage{amssymb}
\usepackage[mathscr]{euscript}
\usepackage{enumerate}
\usepackage{xspace}
\usepackage{color}
\usepackage{enumerate}
\usepackage{enumitem}
\usepackage{amsthm}

\begin{document}

\theoremstyle{plain}
\newtheorem{C}{Convention}
\newtheorem{SA}{Standing Assumption}[section]
\newtheorem{SAS}{Standing Assumption}[subsection]
\newtheorem{theorem}{Theorem}[section]
\newtheorem*{theoremo}{Theorem}
\newtheorem{condition}{Condition}[section]
\newtheorem{lemma}[theorem]{Lemma}
\newtheorem{fact}[theorem]{Fact}
\newtheorem{proposition}[theorem]{Proposition}
\newtheorem{corollary}[theorem]{Corollary}
\newtheorem{claim}[theorem]{Claim}
\newtheorem{definition}[theorem]{Definition}
\newtheorem{Ass}[theorem]{Assumption}
\newcommand{\q}{Q}
\theoremstyle{definition}
\newtheorem{remark}[theorem]{Remark}
\newtheorem{note}[theorem]{Note}
\newtheorem{example}[theorem]{Example}
\newtheorem{assumption}[theorem]{Assumption}
\newtheorem*{notation}{Notation}
\newtheorem*{assuL}{Assumption ($\mathbb{L}$)}
\newtheorem*{assuAC}{Assumption ($\mathbb{AC}$)}
\newtheorem*{assuEM}{Assumption ($\mathbb{EM}$)}
\newtheorem*{assuES}{Assumption ($\mathbb{ES}$)}
\newtheorem*{assuM}{Assumption ($\mathbb{M}$)}
\newtheorem*{assuMM}{Assumption ($\mathbb{M}'$)}
\newtheorem*{assuL1}{Assumption ($\mathbb{L}1$)}
\newtheorem*{assuL2}{Assumption ($\mathbb{L}2$)}
\newtheorem*{assuL3}{Assumption ($\mathbb{L}3$)}
\newtheorem{charact}[theorem]{Characterization}
\newtheorem*{discussion}{Discussion}
\newtheorem*{application}{Application}
\newcommand{\notiz}{\textup} 
\renewenvironment{proof}{{\parindent 0pt \it{ Proof:}}}{\mbox{}\hfill\mbox{$\Box\hspace{-0.5mm}$}\vskip 16pt}
\newenvironment{proofthm}[1]{{\parindent 0pt \it Proof of Theorem #1:}}{\mbox{}\hfill\mbox{$\Box\hspace{-0.5mm}$}\vskip 16pt}
\newenvironment{prooflemma}[1]{{\parindent 0pt \it Proof of Lemma #1:}}{\mbox{}\hfill\mbox{$\Box\hspace{-0.5mm}$}\vskip 16pt}
\newenvironment{proofcor}[1]{{\parindent 0pt \it Proof of Corollary #1:}}{\mbox{}\hfill\mbox{$\Box\hspace{-0.5mm}$}\vskip 16pt}
\newenvironment{proofprop}[1]{{\parindent 0pt \it Proof of Proposition #1:}}{\mbox{}\hfill\mbox{$\Box\hspace{-0.5mm}$}\vskip 16pt}
\newcommand{\s}{\mathfrak{s}}
\newcommand{\X}{X} 
\newcommand{\f}{\mathfrak{f}}
\newcommand{\g}{\mathfrak{g}}
\newcommand{\oP}{P^\circ}
\renewcommand{\c}{\mathfrak{c}}
\renewcommand{\a}{\mathfrak{a}}
\renewcommand{\b}{\mathfrak{b}}
\renewcommand{\C}{\mathsf{C}}
\newcommand{\Y}{Y} 
\newcommand{\z}{\mathfrak{z}}
\renewcommand{\S}{\mathcal{O}}
\newcommand{\xp}{P_{x_0}}
\newcommand{\xq}{Q_{x_0}}

\newcommand{\Law}{\ensuremath{\mathop{\mathrm{Law}}}}
\newcommand{\loc}{{\mathrm{loc}}}
\newcommand{\Log}{\ensuremath{\mathop{\mathscr{L}\mathrm{og}}}}
\newcommand{\Meixner}{\ensuremath{\mathop{\mathrm{Meixner}}}}
\newcommand{\of}{[\hspace{-0.06cm}[}
\newcommand{\gs}{]\hspace{-0.06cm}]}

\let\MID\mid
\renewcommand{\mid}{|}

\let\SETMINUS\setminus
\renewcommand{\setminus}{\backslash}

\def\stackrelboth#1#2#3{\mathrel{\mathop{#2}\limits^{#1}_{#3}}}

\renewcommand{\theequation}{\thesection.\arabic{equation}}
\numberwithin{equation}{section}

\newcommand\llambda{{\mathchoice
      {\lambda\mkern-4.5mu{\raisebox{.4ex}{\scriptsize$\backslash$}}}
      {\lambda\mkern-4.83mu{\raisebox{.4ex}{\scriptsize$\backslash$}}}
      {\lambda\mkern-4.5mu{\raisebox{.2ex}{\footnotesize$\scriptscriptstyle\backslash$}}}
      {\lambda\mkern-5.0mu{\raisebox{.2ex}{\tiny$\scriptscriptstyle\backslash$}}}}}

\newcommand{\prozess}[1][L]{{\ensuremath{#1=(#1_t)_{0\le t\le T}}}\xspace}
\newcommand{\prazess}[1][L]{{\ensuremath{#1=(#1_t)_{0\le t\le T^*}}}\xspace}

\newcommand{\tr}{\operatorname{tr}}
\newcommand{\lijepoa}{{\mathscr{A}}}
\newcommand{\lijepob}{{\mathscr{B}}}
\newcommand{\lijepoc}{{\mathscr{C}}}
\newcommand{\lijepod}{{\mathscr{D}}}
\newcommand{\lijepoe}{{\mathscr{E}}}
\newcommand{\lijepof}{{\mathscr{F}}}
\newcommand{\lijepog}{{\mathscr{G}}}
\newcommand{\lijepok}{{\mathscr{K}}}
\newcommand{\lijepoo}{{\mathscr{O}}}
\newcommand{\lijepop}{{\mathscr{P}}}
\newcommand{\lijepoh}{{\mathscr{H}}}
\newcommand{\lijepom}{{\mathscr{M}}}
\newcommand{\lijepou}{{\mathscr{U}}}
\newcommand{\lijepov}{{\mathscr{V}}}
\newcommand{\lijepoy}{{\mathscr{Y}}}
\newcommand{\cF}{{\mathscr{F}}}
\newcommand{\cG}{{\mathscr{G}}}
\newcommand{\cH}{{\mathscr{H}}}
\newcommand{\cM}{{\mathscr{M}}}
\newcommand{\cD}{{\mathscr{D}}}
\newcommand{\bD}{{\mathbb{D}}}
\newcommand{\bF}{{\mathbb{F}}}
\newcommand{\G}{{\mathbf{G}}}
\newcommand{\bH}{{\mathbb{H}}}
\newcommand{\dd}{\operatorname{d}\hspace{-0.05cm}}
\newcommand{\ddd}{\operatorname{d}}
\newcommand{\er}{{\mathbb{R}}}
\newcommand{\ce}{{\mathbb{C}}}
\newcommand{\erd}{{\mathbb{R}^{d}}}
\newcommand{\en}{{\mathbb{N}}}
\newcommand{\de}{{\mathrm{d}}}
\newcommand{\im}{{\mathrm{i}}}
\newcommand{\indik}{{\mathbf{1}}}
\newcommand{\D}{{\mathbb{D}}}
\newcommand{\E}{E}
\newcommand{\N}{{\mathbb{N}}}
\newcommand{\Q}{{\mathbb{Q}}}
\renewcommand{\P}{{\mathbb{P}}}
\newcommand{\ud}{\operatorname{d}\!}
\newcommand{\ii}{\operatorname{i}\kern -0.8pt}
\newcommand{\cadlag}{c\`adl\`ag }
\newcommand{\p}{P}
\newcommand{\F}{\mathbf{F}}
\newcommand{\1}{\mathbb{I}}
\newcommand{\lle}{\langle\hspace{-0.085cm}\langle}
\newcommand{\rre}{\rangle\hspace{-0.085cm}\rangle}
\newcommand{\llbr}{[\hspace{-0.085cm}[}
\newcommand{\rrbr}{]\hspace{-0.085cm}]}
\newcommand{\oY}{Y^\circ}
\newcommand{\oQ}{P^\circ_{x_0}}

\def\EM{\ensuremath{(\mathbb{EM})}\xspace}

\newcommand{\la}{\langle}
\newcommand{\ra}{\rangle}

\newcommand{\Norml}[1]{%
{|}\kern-.25ex{|}\kern-.25ex{|}#1{|}\kern-.25ex{|}\kern-.25ex{|}}

\title[Absolute Continuity/Singularity of Multidimensional Diffusions]{On Absolute Continuity and Singularity of Multidimensional Diffusions} 
\author[D. Criens]{David Criens}
\address{D. Criens - Technical University of Munich, Center for Mathematics, Germany}
\email{david.criens@tum.de}
\keywords{absolute continuity, singularity, multidimensional diffusion, uniformly integrable martingale, explosion, integral test, perpetual integral, random time change\vspace{1ex}}

\date{\today}
\maketitle

\frenchspacing
\pagestyle{myheadings}

\begin{abstract}
	Consider two laws \(P\) and \(Q\) of multidimensional possibly explosive diffusions with common diffusion coefficient \(\a\) and drift coefficients \(\b\) and \(\b + \a \c\), respectively, and the law \(\oP\) of an auxiliary diffusion with diffusion coefficient  \(\langle \c, \a\c\rangle^{-1}\a\) and drift coefficient \(\langle \c, \a\c\rangle^{-1}\b\).
	We show that \(P \ll Q\) if and only if the auxiliary diffusion \(\oP\) explodes almost surely and that \(P\perp Q\) if and only if the auxiliary diffusion \(\oP\)  almost surely does not explode. As applications we derive a Khasminskii-type integral test for absolute continuity and singularity, an integral test for explosion of time-changed Brownian motion, and we discuss applications to mathematical finance.
\end{abstract}

\section{Introduction}
Consider two laws \(P\) and \(Q\) of multidimensional possibly explosive diffusions with common diffusion coefficient \(\a\) and drift coefficients \(\b\) and \(\b + \a \c\), respectively. We are interested in finding analytic conditions for the absolute continuity \(P \ll Q\) and the singularity \(P \perp Q\). Such conditions are of interest in many branches of probability theory. In mathematical finance, for instance, mutual absolute continuity is of importance in the study of the absence of arbitrage, see \cite{DS,Shir}.

For one-dimensional diffusions precise integral test were proven in \cite{Cherny2006} under the Engelbert--Schmidt conditions.
For multidimensional diffusions the situation is less well-understood and only a few analytic conditions are known, see \cite{BENARI2005179} for an integral test for Fuchsian diffusions.

The starting point for our research is the following probabilistic characterization of absolute continuity and singularity:
Let \(X\) be the coordinate process and set
\begin{align*} 
A_t \triangleq \int_0^{t \wedge \theta} \langle \c(\X_s), \a (\X_s) \c(\X_s)\rangle ds, \quad t \in \mathbb{R}_+, 
\end{align*}
where \(\theta\) is the explosion time.
It has been proven in \cite{BENARI2005179, criensglau2018} that \(P \ll Q\) is equivalent to \(P(A_\theta < \infty) = 1\) and that \(P \perp Q\) is equivalent to \(P(A_\theta = \infty) = 1\). 
In other words, \(P \ll Q\) and \(P\perp Q\) are characterized by \(P\)-a.s. divergence and convergence of the perpetual integral \(A_\theta\).
Again, these properties are well-understood for one-dimensional diffusions, see \cite{Salminen06,CUI2014118,khoshnevisan2006, mijatovic2012, MUSIELA198679}, and it seems that less work has been done for the multidimensional case, see \cite{doi:10.1002/mana.19871310120,10.1007/BFb0083762} for results concerning Bessel processes, \cite{criensglau2018} for some conditions in radial cases, and \cite{kuhn2019} for results on divergence in case \(X\) is a conservative Feller process possibly with jumps.

In \cite{Salminen06,CUI2014118,khoshnevisan2006,kuhn2019} the perpetual integral \(A_\theta\) was related to the hitting time of a time-changed process.
In this article we pick up this idea and prove the following:
Let \(\f \colon \mathbb{R}^d \to (0, \infty)\) be a Borel function which is locally bounded away from zero and infinity. Under the assumptions that the diffusion \(P\) exists and \(\b\) and \(\a\) are locally bounded, we show existence of a diffusion \(\oP\) with diffusion coefficient \(\f^{-1} \a\) and drift coefficient \(\f^{-1} \b\) such that the law of the perpetual integral
\[
T_\theta \triangleq \int_0^\theta \f (\X_s) ds
\]
under \(P\) coincides with the law of the explosion time \(\theta\) under \(\oP\). Furthermore, we show that \(\oP\) is unique  whenever \(P\) is unique.

Returning to our initial problem, we note that in case \(\f = \langle \c, \a \c\rangle\) the absolute continuity \(P \ll Q\) is equivalent to \(\oP(\theta < \infty) = 1\) and the singularity \(P \perp Q\) is equivalent to \(\oP(\theta= \infty) = 1\).
This observation is very useful, because the literature contains many conditions for explosion and non-explosion of multidimensional diffusions, see \cite{mckean1969stochastic,pinsky1995positive,SV}. 
For illustration, we formulate a Khasminskii-type integral test for absolute continuity and singularity. 

The result can also be applied in the converse direction: In case we have criteria for absolute continuity and singularity, these can be used to deduce explosion criteria for time-changed diffusions. To illustrate this, we derive an integral test for almost sure explosion and non-explosion of time-changed Brownian motion, using results on singularity of Fuchsian diffusions proven in \cite{BENARI2005179}. The integral test improves several conditions known in the literature, see \cite{mckean1969stochastic}.

The absolute continuity \(P \ll Q\) is intrinsically connected to the uniform integrable (UI) \(Q\)-martingale property of a certain stochastic exponential (see Eq. \ref{eq: Z} below), which has been studied for one-dimensional diffusions in \cite{MU(2012)}. 
Independent of the dimension, it is known that for the conservative case the loss of the martingale property has a one-to-one relation to the explosion of an auxiliary diffusion, see, e.g., \cite{CFY,Sin}. This turned out to be wrong in the non-conservative setting of \cite{MU(2012)}.  Our result explains that for the UI martingale property the statement is true irrespective whether the diffusions are conservative or non-conservative.

As a third application, we use the relation of the UI martingale property and absolute continuity to study a problem in mathematical finance: We derive an explosion criterion for a local martingale measure to be a true martingale measure on the infinite time horizon.

Let us close the introduction with comments on related literature.
To the best of our knowledge, the relation of absolute continuity/singularity and explosion of a time-changed process has not been reported before.
We think that our new integral tests for absolute continuity/singularity and explosion/non-explosion illustrate that working out this connection is fruitful.  
The integral tests in \cite{Cherny2006,MU(2012)} for absolute continuity, singularity and the UI martingale property in one-dimensional frameworks can be deduced from our result and Feller's test for explosion under additional assumptions on the coefficients. 
For general one-dimensional diffusion models with finite and infinite time horizon, analytic conditions for a local martingale measure to be a martingale measure were given in \cite{MU-det}. 
Conditions for one- and multidimensional diffusion models with finite time horizon were proven in \cite{doi:10.1142/S0219024918500024, criens20,Sin}. We extend part of these results to a multidimensional setting with infinite time horizon.
Beginning with \cite{doi:10.1137/1103025}, existence and uniqueness results for time-changed Markov processes have a long history, see, e.g., \cite{bottcher2014levy,SV} for more information. In most of the classical work, the function \(\f\) is assumed to be uniformly bounded away from zero, which implies that the time-changed process is conservative in case the original process is conservative. 
In combination with conditions for non-explosion, general positive continuous \(\f\) are considered in the recent article \cite{kuhn2019}. 
The novelty of our existence and uniqueness result is that we work without additional assumptions for non-explosion. This is crucial for the question of absolute continuity and singularity. Moreover, we work under sort of minimal assumptions on the original diffusion \(P\) by assuming only existence and locally bounded coefficients.

The article is structured as follows: In Section \ref{sec: MR} we present our main results, in Section \ref{sec: appl} we discuss applications and in Section \ref{sec: pf main} we prove our main theorem.

\section{Main Results}\label{sec: MR}
Let \(\mathbb{R}^d_\Delta\triangleq \mathbb{R}^d \cup \{\Delta\}\) be the one-point compactification of \(\mathbb{R}^d\). We define \(\Omega\) be the space of all continuous functions \(\mathbb{R}_+ \to \mathbb{R}^d_\Delta\) which are absorbed in \(\Delta\).
Let \(\X\) be the coordinate process on \(\Omega\), i.e. \(\X_t(\omega) = \omega(t)\) for all \(\omega \in \Omega\) and \(t \in \mathbb{R}_+\), and define \(\mathcal{F} \triangleq \sigma (\X_t, t \in \mathbb{R}_+), \mathcal{F}^o_t \triangleq \sigma (\X_s, s \in [0, t])\) and \(\mathcal{F}_t \triangleq \mathcal{F}^o_{t +} \triangleq \bigcap_{s > t} \mathcal{F}^o_s\) for all \(t \in \mathbb{R}_+\). Except stated otherwise, all terms such as \emph{martingale, local martingale, stopping time} etc. correspond to \(\F \triangleq (\mathcal{F}_t)_{t \geq 0}\) as underlying filtration.
Let \(\mathbb{S}^{d \times d}\) be the space of symmetric non-negative definite real-valued \(d  \times d\) matrices.

For \(n \in \mathbb{N}\) we set 
\[
\theta_n \triangleq \inf (t \in \mathbb{R}_+ \colon \|\X_t\| \geq n), \quad \theta \triangleq \inf(t \in \mathbb{R}_+ \colon X_t = \Delta) = \lim_{m \to \infty} \theta_m.
\]
It is well-known that \(\theta_n\) and \(\theta\) are stopping times.
We fix two Borel functions \(\b \colon \mathbb{R}^d \to \mathbb{R}^d\) and \(\a \colon \mathbb{R}^d \to \mathbb{S}^{d \times d}\) and impose the following:
\begin{enumerate}
\item[(S1)] \(\b\) and \(\a\) are locally bounded.
\end{enumerate}
Here, \(S\) is an acronym for \emph{standing}, which indicates that the assumption is in force for the remainder of the section.

The following definition of a \emph{martingale problem} is taken from \cite{pinsky1995positive}, where it is called \emph{generalized martingale problem} due to the possibility of explosion. For simplicity we drop the term \emph{generalized}.
\begin{definition}
	We say that a probability measure \(P\) on \((\Omega, \mathcal{F})\) solves the \emph{martingale problem} MP \((\a, \b, x_0)\), where \(x_0 \in \mathbb{R}^d\), if \(P(\X_0 = x_0) = 1\) and for all \(n \in \mathbb{N}\) and \(f \in C^2(\mathbb{R}^d)\) the process
	\[
	f(\X_{\cdot \wedge \theta_n}) - f(x_0) - \int_0^{\cdot \wedge \theta_n} \big( \langle \nabla f (\X_s), \b(\X_s)\rangle + \tfrac{1}{2} \textup{tr} (\nabla^2 f (\X_s)\a (\X_s))\big)ds
	\]
	is a \(P\)-martingale. A solution \(P\) is called \emph{conservative} (or \emph{non-explosive}), if \(P(\theta = \infty) = 1\),
	and \emph{almost surely explosive}, if \(P(\theta < \infty) = 1\).
\end{definition}

The following theorem is the key observation in this article.
It explains that perpetual integrals are distributed as the explosion time of a time-changed diffusion. 
The proof is given in Section \ref{sec: pf main} below. 
\begin{theorem}\label{theo: tc}
	Let \(\f \colon \mathbb{R}^d \to (0, \infty)\) be Borel and locally bounded away from zero and infinity. There exist \(\mathbb{R}^d_\Delta\)-valued right-continuous measurable processes \(Y\) and \(U\) such that for every \(x_0 \in \mathbb{R}^d\) and every solution \(P_{x_0}\) to the MP \((\a, \b, x_0)\) the following hold:
	\begin{enumerate}
		\item[\textup{(i)}] \(Y\) has \(\xp\)-a.s. continuous paths.
			\item[\textup{(ii)}]  \(\xp\)-a.s. \(U \circ Y = X\).
		\item[\textup{(iii)}] \(\oP_{x_0} \triangleq \xp \circ Y^{-1}\) solves the MP \((\f^{-1} \a, \f^{-1} \b, x_0)\).
		\item[\textup{(iv)}] For all Borel sets \(A \subseteq [0, \infty]\)
			\begin{align}\label{eq: exp id}
		\xp \Big(\int_0^\theta \f(\X_s) ds  \in A\Big) = \oP_{x_0}(\theta  \in A).
		\end{align}
	\end{enumerate}	
In particular, existence and uniqueness hold simultaneously for the MPs \((\a, \b, x_0)\) and \((\f^{-1}\a, \f^{-1} \b, x_0)\).
\end{theorem}

\begin{remark}
	While our main interest lies in the equality \eqref{eq: exp id}, also the existence and uniqueness parts of Theorem \ref{theo: tc} are useful, because they lead to localizations of known existence and uniqueness theorems. For example, suppose that \(\textup{det}(\a)\) is locally bounded away from zero. This is a weak ellipticity assumption without continuity. There exists a continuous function \(\f \colon \mathbb{R}^d \to (0, \infty)\) such that \(\f \b\) and \(\f \a\) are bounded. Thus, by an existence result from \cite{ROZKOSZ1991187}, the MP \((\f \a, \f \b, x_0)\) has a (conservative) solution. 
	Now, Theorem \ref{theo: tc} implies that the MP \((\a, \b, x_0)\) has a (not necessarily conservative) solution, too. We give more details on this in Appendix \ref{app}.
\end{remark}

Fix a third Borel function \(\c \colon \mathbb{R}^d \to \mathbb{R}^d\) with the following property:
\begin{enumerate}
	\item[(S2)] \(\a \c\) is locally bounded.
\end{enumerate}

Before we turn to our main application, we report a simple observation which might be of independent interest: Absolute continuity and singularity are invariant under time-changes. 
For an application of a related result to mathematical finance see \cite{doi:10.1111/j.1467-9965.2012.00530.x}.
\begin{corollary}\label{coro: inv}
	Let \(\f \colon \mathbb{R}^d \to (0, \infty)\) be a  Borel function locally bounded away from zero and infinity, \(x_0 \in \mathbb{R}^d\), \(\xp\) be the unique solution to the MP \((\a, \b, x_0)\), \(\xq\) be the unique solution to the MP \((\a, \b + \a \c, x_0)\), \(\xp^*\) be the unique solution to the MP \((\f^{-1} \a, \f^{-1} \b, x_0)\) and let \(\xq^*\) be the unique solution to the MP \((\f^{-1} \a, \f^{-1} (\b + \a \c), x_0)\). The following hold:
	\begin{enumerate}
		\item[\textup{(i)}] \(\xp \ll \xq\) if and only if \(\xp^* \ll \xq^*\).
		\item[\textup{(ii)}] \(\xp \perp \xq\) if and only if \(\xp^* \perp \xq^*\).
	\end{enumerate}
\end{corollary}
\begin{proof}
	Assume that \(\xp \ll \xq\). For contradiction, let \(A \in \mathcal{F}\) be such that \(\xq^*(A) = 0\) and \(\xp^*(A) > 0\). Set \(B \triangleq \{Y \in A\}\), where \(Y\) is as in Theorem \ref{theo: tc}. Then, Theorem \ref{theo: tc} yields that
	\(
	\xq (B) = \xq^*(A) = 0\) and \(\xp (B) = \xp^* (A) > 0. 
	\)
	This is a contradiction and we conclude that \(\xp^* \ll \xq^*\). The converse implication in part (i) follows by symmetry. Part (ii) can be shown in the same manner.
\end{proof}

We assume the following:
\begin{enumerate}
	\item[(S3)] For every \(x_0 \in \mathbb{R}^d\) there exists a unique solution \(\xp\) to the MP \((\a, \b, x_0)\) and a unique solution \(\xq\) to the MP \((\a, \b + \a \c, x_0)\). 
\end{enumerate}
Analytic conditions for (S3) are given by Proposition \ref{prop: EU} in Appendix \ref{app}.

Next, we introduce a non-negative local \(\xq\)-martingale which relates \(\xq\) and \(\xp\).  For this, we assume the following:
\begin{enumerate}
	\item[\textup{(S4)}] \(\langle \c, \a\c\rangle\) is locally bounded.
\end{enumerate}
We set
\[
\overline{\X}_{\cdot \wedge \theta_n} \triangleq \X_{\cdot \wedge \theta_n} - \int_0^{\cdot \wedge \theta_n} \big(\b(\X_s) + \a (\X_s) \c (\X_s) \big) ds.
\]
By definition of the martingale problem, \(\overline{\X}_{\cdot \wedge \theta_n}\) is a continuous \(\xq\)-martingale with quadratic variation process
\[
[\overline{\X}_{\cdot \wedge \theta_n}, \overline{\X}_{\cdot \wedge \theta_n}] = \int_0^{\cdot \wedge \theta_n} \a(\X_s) ds.
\]
By assumption (S4), the integral process \(\overline{Y}_{\cdot \wedge \theta_n} \triangleq \int_0^{\cdot \wedge \theta_n} \langle\c(\X_s), d \overline{\X}_s\rangle\) is well-defined as a continuous \(\xq\)-martingale starting at zero with quadratic variation process 
\[
[\overline{Y}_{\cdot \wedge \theta_n}, \overline{Y}_{\cdot \wedge \theta_n}] = \int_0^{\cdot \wedge \theta_n} \langle \c(\X_s), \a(\X_s) \c(\X_s)\rangle ds.
\]
\begin{lemma}
	The process 
	\begin{align}\label{eq: Z}
	Z_t \triangleq \begin{cases} \exp \big( - \int_0^t \langle \c(\X_s), d \overline{\X}_s\rangle - \frac{1}{2} \int_0^t \langle \c(\X_s), \a(\X_s) \c(\X_s)\rangle ds \big), & t < \theta,\\
	\liminf_{n \to \infty} Z_{\theta_n},  & t \geq \theta. \end{cases}
	\end{align}
	is a non-negative local \(\xq\)-martingale and \(\xq\)-a.s.  the terminal value \(Z_\infty \triangleq \lim_{t \to \infty} Z_t\) exists and is finite.
\end{lemma}
\begin{proof}
This follows similar to the proof of \cite[Lemma 12.43]{J79}. 
\end{proof}

As in the introduction, we set
\[
A_\theta \triangleq \int_0^\theta \langle \c (X_s), \a (X_s) \c (X_s) \rangle ds.
\]
Part (i) of the next proposition is a version of \cite[Corollary 5.1]{criensglau2018} (see also \cite[Theorem 3.3]{RufSDE}) and part (ii) is an extension of \cite[Theorem 1 (i)]{BENARI2005179} to a non-conservative setting. The proofs are similar and omitted.
\begin{proposition}\label{prop: ac sin}
	\begin{enumerate}
		\item[\textup{(i)}] The following are equivalent:
			\begin{enumerate}
			\item[\textup{(a)}] \(\xp \ll \xq\) with \(\frac{d\xp}{d\xq} = Z_\infty\).
			\item[\textup{(b)}] \(Z\) is a uniformly integrable \(\xq\)-martingale.
			\item[\textup{(c)}] \(\xp(A_\theta < \infty) = 1\).
		\end{enumerate}
		\item[\textup{(ii)}] The following are equivalent:
			\begin{enumerate}
			\item[\textup{(a)}] \(\xp \perp \xq\).
			\item[\textup{(b)}] \(\xp(A_\theta = \infty) = 1\).
		\end{enumerate}
	\end{enumerate}
\end{proposition}

From now on we extend (S4) and assume the following:
\begin{enumerate}
	\item[(S5)] \(\langle \c, \a \c\rangle\) is locally bounded away from zero.
\end{enumerate}

Due to Theorem \ref{theo: tc} there exists a unique solution \(\oQ\) to the time-changed MP \((\langle \c, \a \c\rangle^{-1} \a, \langle \c, \a \c\rangle^{-1} \b, x_0)\) and 
\[
\xp (A_\theta < \infty) = \oQ(\theta < \infty).
\]
This observation relates Proposition \ref{prop: ac sin} to \(\oQ\):
\begin{corollary}\label{theo: abstract main}
	\begin{enumerate}
		\item[\textup{(i)}] \textup{(i.a) -- (i.c)} in Proposition \ref{prop: ac sin} are equivalent to the following:
		\begin{enumerate}\item[\textup{(d)}] \(\oQ(\theta < \infty) = 1\).
			\end{enumerate}
				\item[\textup{(ii)}] \textup{(ii.a)} and \textup{(ii.b)} in Proposition \ref{prop: ac sin} are equivalent to the following:
					\begin{enumerate}\item[\textup{(c)}] \(\oQ(\theta = \infty) = 1\).
				\end{enumerate}
	\end{enumerate}
\end{corollary}
\begin{remark}\label{rem: tmg}
	Due to \cite[Corollary 5.1]{criensglau2018}, the following are equivalent:
	\begin{enumerate}
		\item[\textup{(i)}] \(\xp \ll_\textup{loc} \xq\), i.e. \(\xp \ll \xq\) on \(\mathcal{F}_t\) for all \(t \in \mathbb{R}_+\).
		\item[\textup{(ii)}] \(Z\) is a \(\xq\)-martingale.
		\item[\textup{(iii)}] \(\xp\)-a.s. \(A_\theta < \infty\) on \(\{\theta < \infty\}\).
	\end{enumerate} 
	We note that if \(\xp\) is \emph{non-explosive}, then \(\xp \ll_\textup{loc} \xq\) and \(Z\) is a \(\xq\)-martingale.
	In comparison, the absolute continuity \(\xp \ll \xq\) and the UI \(\xq\)-martingale property of \(Z\) are related to almost sure \emph{explosion} of \(\oQ\). 
\end{remark}

Let us shortly comment on the role played by the initial value.

\begin{lemma}\label{lem: MP}
	Suppose that \(\a\) is locally H\"older continuous and \(\langle \xi, \a(x)\xi\rangle > 0\) for all \(x \in \mathbb{R}^d\) and \(\xi \in \mathbb{R}^d \backslash \{0\}\). Then, the following hold:
	\begin{enumerate}
		\item[\textup{(i)}] 
		\(\xp (\theta = \infty) = 1\) holds for all \(x_0 \in \mathbb{R}^d\) if it holds for some \(x_0 \in \mathbb{R}^d\).
		\item[\textup{(ii)}]
				\(\xp (\theta < \infty) = 1\) holds for all \(x_0 \in \mathbb{R}^d\) if it holds for some \(x_0 \in \mathbb{R}^d\).
	\end{enumerate}
\end{lemma}
\begin{proof}
	This follows from the fact that \(x \mapsto P_x(\theta < \infty)\) is harmonic and the maximum principle, see \cite[Lemmata 1.2, 1.4]{BHATTACHARYA198295}.
\end{proof}

\begin{corollary}
	Suppose that \(\a\) satisfies the assumptions from Lemma \ref{lem: MP} and that \(\langle \c, \a \c\rangle\) is locally H\"older continuous. Then the following hold:
		\begin{enumerate}
		\item[\textup{(i)}] 
		\(\xp \ll \xq\) holds for all \(x_0 \in \mathbb{R}^d\) if it holds for some \(x_0 \in \mathbb{R}^d\).
		\item[\textup{(ii)}]
		\(\xp \perp \xq\) holds for all \(x_0 \in \mathbb{R}^d\) if it holds for some \(x_0 \in \mathbb{R}^d\).
	\end{enumerate}
\end{corollary}
For conservative martingale problems this observation has been reported in \cite[Corollary 1]{BENARI2005179}.

Many analytic conditions for almost sure explosion and non-explosion of martingale problems are known, see, e.g., \cite{mckean1969stochastic, pinsky1995positive,SV}.
Due to Corollary \ref{theo: abstract main}, these lead to analytic conditions for \(\xp \ll \xq\) and \(\xp \perp \xq\).

For the one-dimensional case an analytic characterization of \(\oQ(\theta  < \infty) \in \{0, 1\}\) is given in \cite[Theorem 5.5.29, Proposition 5.5.32]{KaraShre}. Together with Corollary \ref{theo: abstract main} it leads to the integral tests in \cite[Corollaries 5.1, 5.3]{Cherny2006} and \cite[Theorem 2.3]{MU(2012)}. 
Our time-change argument requires that \(\c \not = 0\), which is not needed in \cite{Cherny2006,MU(2012)}. In return, our approach is robust w.r.t. the dimension.

In the following section we present three applications of Theorem \ref{theo: tc} and Corollary \ref{theo: abstract main}. First, we derive a Khasminskii-type integral test for absolute continuity and singularity, second we derive a Feller-type integral test for explosion of a multidimensional time-changed Brownian motion, and third we outline an application of Corollary \ref{theo: abstract main} to mathematical finance in which \(\c \not = 0\) is usually naturally satisfied.

\section{Three Applications}\label{sec: appl}
\subsection{A Khasminskii-Test for Absolute Continuity/Singularity}
In this section we assume that (S1) -- (S5) from Section \ref{sec: MR} hold. We now formulate analytic conditions for \(\xp \ll \xq\) and \(\xp \perp \xq\).
\begin{condition}\label{cond: eq}
	There exist continuous functions \(B \colon [\frac{1}{2}, \infty) \to \mathbb{R}\) and \(A \colon [\frac{1}{2}, \infty) \to (0, \infty)\) such that for all \(x \in \mathbb{R}^d \colon  \|x\| \geq 1\)
	\begin{align*}
	A \Big( \frac{\|x\|^2}{2} \Big) &\leq \frac{\langle x,\a (x) x \rangle}{\langle \c(x), \a (x) \c(x)\rangle},\\
	\langle x, \a (x) x \rangle B \Big( \frac{\|x\|^2}{2} \Big) &\leq  \textup{tr} (\a (x) ) + 2 \langle x, \b (x) \rangle,  
	\end{align*}
	and 
	\[
	\int_{\frac{1}{2}}^\infty \frac{1}{C(z)} \int_{\frac{1}{2}}^z \frac{C(u)}{A(u)} dudz < \infty, 
	\]
	where
	\[
	C(z) \triangleq \exp \Big(\int_1^z B(u) du\Big).
	\]
\end{condition}
\begin{condition}\label{cond: non eq}
There exists an \(R> 0\) and continuous functions \(B \colon [R, \infty) \to \mathbb{R}\) and \(A \colon [R, \infty) \to (0, \infty)\) such that for all \(x \in \mathbb{R}^d \colon \|x\| \geq \sqrt{2R}\)
\begin{align*}
A \Big( \frac{\|x\|^2}{2} \Big) &\geq \frac{\langle x,\a (x) x \rangle}{\langle \c(x),\a (x) \c(x)\rangle},\\
\langle x, \a(x) x \rangle B \Big( \frac{\|x\|^2}{2} \Big) &\geq  \textup{tr} ( \a (x) ) + 2 \langle x, \b (x) \rangle,  
\end{align*}
and 
\[
\int_R^\infty \frac{1}{C(z)} \int_R^z \frac{C(u)}{A(u)} dudz = \infty, 
\]
where
\[
C(z) \triangleq \exp \Big(\int_R^z B(u) du\Big).
\]
\end{condition}
\begin{corollary}\label{theo}
	\begin{enumerate}
		\item[\textup{(i)}] Suppose that Condition \ref{cond: eq} holds.  Then, \(\xp \ll \xq\) with \(\frac{d \xp}{d \xq} = Z_\infty\). In particular, \(Z\) is a uniformly integrable \(\xq\)-martingale.
		\item[\textup{(ii)}] Suppose that  Conditions \ref{cond: non eq} holds.  Then, \(\xp \perp \xq\) and \(Z\) is no uniformly integrable \(\xq\)-martingale.
	\end{enumerate}
\end{corollary}
\begin{proof}
	Due to \cite[Theorem 10.2.4]{SV}, Condition \ref{cond: eq} implies that \(\oQ(\theta < \infty) = 1\).
In case Condition \ref{cond: non eq} holds, \cite[Theorem 10.2.3]{SV} yields that \(\oQ(\theta = \infty) = 1\).
Now, all claims follow from Corollary \ref{theo: abstract main}. 
\end{proof}

\subsection{An Explosion-Test for Time-Changed Brownian Motion}
Let \(\g \colon \mathbb{R}^d \to (0, \infty)\) be a Borel function which is locally bounded away from zero and infinity.
Due to Theorem \ref{theo: tc}, for every \(x_0 \in \mathbb{R}^d\) there exists a unique solution \(\xp\) to the MP \((\g \textup{ Id}, 0, x_0)\).
In case \(d \leq 2\), it follows immediately from Theorem \ref{theo: tc} and \cite[Theorem 3.27]{morters_peres_2010} that \(\xp\) is non-explosive. 
We are interested in explosion properties of \(\xp\) in case  \(d \geq 3\), i.e. in the situation where Brownian motion is transient.

For the remainder of this section let \(d \geq 3\) and denote by \(\mathcal{W}_{x_0}\) the \(d\)-dimensional Wiener measure with initial value \(x_0\). Then, \cite[Theorems 3.32, 3.33]{morters_peres_2010} yield that
\[
E^{\mathcal{W}_{x_0}} \Big[ \int_0^\infty \frac{ds}{\g (X_s)} \Big] = \C_d\int_{\mathbb{R}^d} \frac{\|x - x_0\|^{2- d} dx}{\g(x)},
\]
for a dimension-dependent constant \(\C_d > 0\).
This observation and Theorem \ref{theo: tc} imply the following:
\begin{corollary}\label{coro: sure expl}
	If \(\int_{\mathbb{R}^d} \g^{-1}(x) \|x - x_0\|^{2 - d} dx < \infty\), then \(\xp(\theta < \infty) = 1\).
\end{corollary}
By the standard linear growth condition for non-explosion, we have \(\xp (\theta = \infty) = 1\) in case
\[
\g (x) \leq \C(1 + \|x\|)^2, \quad x \in \mathbb{R}^d, \C> 0.
\]
The following corollary shows that in case \(\g\) is locally H\"older continuous and a least of quadratic growth the convergence criterion in Corollary \ref{coro: sure expl} is optimal.
\begin{corollary}\label{coro: no expl} 
	Suppose that \(\g\) is locally H\"older continuous and
	\begin{align}\label{eq: growth}
	\g (x) \geq \C (1 + \|x\|)^2, \qquad x \in \mathbb{R}^d, \C> 0.
	\end{align}
If \(\int_{\mathbb{R}^d} \g^{-1} (x) \|x\|^{2 - d} dx = \infty\), then \(\xp (\theta = \infty) = 1\).
\end{corollary}
\begin{proof}
	We define \(\a \triangleq \textup{Id}, \b \triangleq 0\) and \(\c \triangleq \g^{- \frac{1}{2}} e_1\), where \(e_1\) is the first unit vector. Let \(\xq\) be the unique solution to the MP \((\a, \c, x_0)\), see Proposition \ref{prop: EU} in Appendix \ref{app}. Note that \(\langle \c,  \a \c\rangle = \g^{-1}\) is a strictly positive continuous function. Corollary \ref{theo: abstract main} yields that \(\xp (\theta = \infty) = 1\) if and only if \(\mathcal{W}_{x_0} \perp \xq\). It follows from \cite[Corollary 4]{BENARI2005179}\footnote{The statement of \cite[Corollary 4]{BENARI2005179} contains a small typo: \(|b(x)|\) has to be replaced by \(|b(x)|^2\), see \cite[Eq. 1.2]{BENARI2005179}.} that 
	\[
	\mathcal{W}_{x_0} \perp \xq \Leftrightarrow \int_{\mathbb{R}^d} \frac{\|x\|^{2 - d} dx}{\g(x)} = \infty.
	\]
	This completes the proof.
\end{proof}
\begin{remark}
	The growth condition \eqref{eq: growth} and \(\int_{\mathbb{R}^d} \g^{-1} (x) \|x\|^{2 - d} dx = \infty\) do not exclude themselves: In case \eqref{eq: growth} holds, we have
	\[
	\int_{\mathbb{R}^d} \g^{-1} (x) \|x\|^{2 - d} dx \leq \C \int_{\mathbb{R}^d} \frac{\|x\|^{2 - d} dx}{1 + \|x\|^2} = \C_d \int_0^\infty \frac{r dr}{1 + r^2} = \infty.
	\]
\end{remark}

The following proposition explains that in general the growth condition \eqref{eq: growth} is sharp. 
\begin{proposition}
	Let \(\rho\colon \mathbb{R}_+ \to [1, \infty)\) be an increasing function with \(\rho (0) = 1\) and \(\rho (x) \to \infty\) as \(x \to \infty\). There exists a function \(\g\) such that the following hold:
	\begin{enumerate}
		\item[\textup{(i)}] \(\g(x) \geq \frac{1 + \|x\|^2}{\rho (\|x\|)}\) for all \(x \in \mathbb{R}^d\).
				\item[\textup{(ii)}] \(P_0(\theta < \infty) = 1\).
						\item[\textup{(iii)}] \(\int_{\mathbb{R}^d} \g^{-1} (x) \|x\|^{2 - d} dx = \infty\).
	\end{enumerate}
\end{proposition}
\begin{proof}
	We adapt the proof of \cite[Theorem 3]{BENARI2005179}. 
	Let \(e_1\) be the first unit vector, set \(x_1 \triangleq e_1\) and define inductively
	\begin{equation}  \label{eq: growth example}
	\begin{split}
	R_n &\triangleq 3^{-n} \|x_n\|,
	\\
	x_{n + 1} &\in \big\{t e_1 \colon t > 4 \|x_n\|, \rho (\tfrac{t}{2}) > 4^{d (n + 1)}\big\}.
	\end{split}\end{equation}
	Set \(B_{R} (x) \triangleq \{y  \in \mathbb{R}^d \colon \|x - y\| < R\}\) and note that the balls \((B_{R_n} (x_n))_{n \in \mathbb{N}}\) are disjoint, because
	\[
	\|x_{n + 1}\| - \|x_n\| > \frac{3 \|x_{n + 1}\|}{4} = \frac{3^{n + 2} R_{n + 1}}{4} > \frac{9}{8}\big(R_{n + 1} + R_{n}\big),
	\] 
	where we use \eqref{eq: growth example} and in particular that \(3R_{n + 1} > 4R_n\).
	Define 
	\[
	\g (x) \triangleq \begin{cases} 
	\frac{1 + \|x\|^2}{\rho (\|x_n\| - R_n)},& x \in B_{R_n} (x_n) \text{ for some }n \in \mathbb{N},\\
	2 + \|x\|^4,&x \not \in \bigcup_{n \in \mathbb{N}} B_{R_n} (x_n) \triangleq G.
	\end{cases}
	\]
It is clear that \(\g\) is Borel and locally bounded away from zero and infinity.
	If \(x \in B_{R_n} (x_n)\) we have \(\|x_n\| - R_n \leq \|x\|\) and 
	\[
	\frac{\g (x) \rho (\|x\|)}{1 + \|x\|^2} = \frac{\rho (\|x\|)}{\rho(\|x_n\| - R_n)} \geq  1,
	\]
	because \(\rho\) is increasing. 
	If \(x \not \in G = \bigcup_{n \in \mathbb{N}} B_{R_n} (x_n)\) we have
	\[
	\frac{\g (x) \rho (\|x\|)}{1 + \|x\|^2} \geq \rho (\|x\|) \geq 1.
	\]
	In other words, (i) holds.
	
	Next, we show (ii). Due to \cite[Corollary 3.19]{morters_peres_2010} we have 
	\[
	\sum_{n = 1}^\infty \mathcal{W}_0(X \textup{ hits } B_{R_n} (x_n)) = \sum_{n = 1}^\infty \Big(\frac{R_n}{\|x_n\|}\Big)^{d - 2} = \sum_{n = 1}^\infty 3^{- n (d + 2)} < \infty.
	\]
	Thus, the Borel--Cantelli lemma yields that \(\mathcal{W}_0\)-a.a. paths of \(X\) hit only finitely many elements of \((B_{R_n}(x_n))_{n \in \mathbb{N}}\).
	Recalling that Brownian motion is transient for \(d \geq 3\), i.e. that \(\mathcal{W}_0\)-a.a. paths of \(X\) leave bounded domains forever in finite time, we conclude that \(\mathcal{W}_0\)-a.s.
	\[
	\int_0^\infty \g^{-1} (X_s) \1_{G} (X_s) ds < \infty.
	\]
	Note that 
	\[
	E^{\mathcal{W}_0} \Big[ \int_0^\infty \frac{ds}{2 + \|X_s\|^4} \Big] = \int_{\mathbb{R}^d} \frac{\|x\|^{2 - d} dx}{2 + \|x\|^4}	= d \omega_d \int_0^\infty \frac{r dr}{2 + r^4} < \infty, 
	\]
	where \(\omega_d\) is the volume of the unit ball in \(\mathbb{R}^d\).
	We conclude that \(\mathcal{W}_0\)-a.s.
	\[
	\int_0^\infty \g^{-1} (X_s) ds < \infty.
	\] 
	Thus, Theorem \ref{theo: tc} yields that \(P_0 (\theta< \infty) = 1\), i.e. (ii) holds. 
	
	It is left to verify (iii).
	Using \eqref{eq: growth example}, the fact that \(f(x) = \|x\|^{2 - d}\) is harmonic on \(\mathbb{R}^d\) and the mean-value property of harmonic functions, we obtain
	\begin{align*}
	\int_{\mathbb{R}^d} \frac{\|x\|^{2 - d} dx}{\g(x)} &\geq \sum_{n = 1}^\infty \int_{B_{R_n} (x_n)} \frac{\|x\|^{2 - d} dx}{\g(x)}
	\\&= \sum_{n = 1}^\infty \rho (\|x_n\| - R_n) \int_{B_{R_n} (x_n)} \frac{\|x\|^{2 - d} dx}{1 + \|x\|^2}
		\\&\geq  \omega_d \sum_{n = 1}^\infty \rho (\|x_n\| - R_n) \frac{\|x_n\|^{2- d} R_n^{d}}{1 + (\|x_n\| + R_n)^2} 
				\\&\geq  \omega_d \sum_{n = 1}^\infty \frac{\rho ((1 - 3^{-n})\|x_n\|) }{1 + (1 + 3^{-n})^2}  \Big(\frac{R_n}{\|x_n\|}\Big)^d
				\\&\geq \frac{\omega_d}{5} \sum_{n = 1}^\infty \rho \Big(\frac{\|x_n\|}{2}\Big) 3^{- dn}
				\\&\geq \frac{\omega_d}{5} \sum_{n = 1}^\infty \Big(\frac{4}{3}\Big)^{- dn} = \infty.
	\end{align*}
	This implies (iii) and the proof is complete.
\end{proof}

In case \(\g\) is radially symmetric, the growth condition on \(\g\) is not needed:
\begin{corollary}\label{coro: radial}
	Suppose that \(\g (x) = \s (\|x\|)\) for a Borel function \(\s \colon \mathbb{R}_+ \to (0, \infty)\) which is locally bounded away from zero and infinity. The following hold:
	\begin{enumerate}
		\item[\textup{(i)}] If \(\int_{x_0}^\infty r \s^{-1} (r) dr < \infty\), then \(\xp (\theta < \infty) = 1\).
			\item[\textup{(ii)}] If \(\int_{x_0}^\infty r \s^{-1} (r) dr = \infty\), then \(\xp (\theta = \infty) = 1\).
	\end{enumerate}
\end{corollary}
\begin{proof}
	Due to \cite[Theorem 2]{doi:10.1002/mana.19871310120} and \cite[Corollary 3]{10.1007/BFb0083762}, for every Borel function \(\z \colon \mathbb{R}_+ \to \mathbb{R}_+\)
		the following are equivalent:
		\begin{enumerate}
			\item[(a)] \(\mathcal{W}_{x_0}(\int_0^\infty \z(\|X_s\|) ds < \infty) > 0\).
					\item[(b)] \(\mathcal{W}_{x_0}(\int_0^\infty \z(\|X_s\|) ds < \infty) = 1\).
							\item[(c)] \(\int_{x_0}^\infty z \z(z) dz < \infty\).
	\end{enumerate}
The claims now follow directly from Theorem \ref{theo: tc}.
\end{proof}

\subsection{On the Absence of Arbitrage in Diffusion Markets}
Suppose that \(\a\) is continuous and strictly positive definite. Then, there exists a unique solution \(\xq\) to the MP \((\a, 0, x_0)\) due to Proposition \ref{prop: EU} in Appendix \ref{app}.
In addition, we assume that \(\xq\) is conservative and we define a process \(S = (S^1, \dots, S^d)\) by
\[
S^i \triangleq \exp \big( \X^i - x_0^i - \tfrac{1}{2} [\X^i, \X^i]\big), \quad i = 1, \dots, d.
\]
We think of \(S\) as discounted price process in a financial market with \(d\) risky assets. The assumption that \(\a\) is strictly positive definite corresponds to the assumption that \(S\) has non-vanishing volatility, which is a typical assumption in mathematical finance.
Because \(\xq\) solves a martingale problem with zero drift, the process \(S\) is a non-negative local \(\xq\)-martingale. Hence, we call \(\xq\) a \emph{local martingale measure}. We are interested whether \(S\) is a UI \(\xq\)-martingale, in which case we call \(\xq\) a \emph{martingale measure}. 
This question is of importance in mathematical finance as it determines the existence or absence of certain arbitrage opportunities, see \cite{DS,Shir} for more details.
We note that \(S^i\) equals \(Z\) as defined in \eqref{eq: Z} for \(- \c\) set to be the \(i\)-th unit vector \(e_i\). Thus, with \(\b \triangleq \a e_i\) we are in the setting of Section \ref{sec: MR}. In particular, (S1) -- (S5) hold by Proposition \ref{prop: EU} in Appendix \ref{app} and the assumptions on \(\a\).
Consequently, Corollary \ref{theo: abstract main} implies the following:
\begin{corollary}\label{coro: MM} \(S^i\) is a UI \(\xq\)-martingale if and only if \(Q^i_{x_0} (\theta < \infty) = 1\), where \(Q^i_{x_0}\) is the unique solution to the MP \((\a_{ii}^{-1} \a, \a^{-1}_{ii} \a e_i, x_0)\). 
\end{corollary}	Applying this corollary for all \(i = 1, \dots, d\), we obtained explosion criteria for \(\xq\) to be a martingale measure. 
Based on results from \cite{mckean1969stochastic,pinsky1995positive,SV} one can also formulate analytic conditions. We leave the formulation of such conditions to the reader.

Finally, let us stress that the results for finite and infinite time horizons are very different.
For example, in case \(d = 1\), the probability measure \(Q^1_{x_0}\) solves the MP \((1, 1, x_0)\), which is obviously conservative, and \(S = S^1\) is a no UI \(\xq\)-martingale,  while it is a \(\xq\)-martingale if and only if \(\int_0^\infty \frac{dx}{\a(x)} = \infty\), see \cite[Proposition 5.2]{doi:10.1142/S0219024918500024}.

\section{Proof of Theorem \ref{theo: tc}} \label{sec: pf main}
In this section we prove Theorem \ref{theo: tc}, i.e. we prove the following:
\begin{theoremo} 
	Let \(\f \colon \mathbb{R}^d \to (0, \infty)\) be Borel and locally bounded away from zero and infinity. There exist \(\mathbb{R}^d_\Delta\)-valued right-continuous measurable processes \(Y\) and \(U\) such that for every \(x_0 \in \mathbb{R}^d\) and every solution \(P_{x_0}\) to the MP \((\a, \b, x_0)\) the following hold:
	\begin{enumerate}
		\item[\textup{(i)}] \(Y\) has \(\xp\)-a.s. continuous paths.
		\item[\textup{(ii)}]  \(\xp\)-a.s. \(U \circ Y = X\).
		\item[\textup{(iii)}] \(\oP_{x_0} \triangleq \xp \circ Y^{-1}\) solves the MP \((\f^{-1} \a, \f^{-1} \b, x_0)\).
		\item[\textup{(iv)}] For all Borel sets \(A \subseteq [0, \infty]\)
		\begin{align*} 
		\xp \Big(\int_0^\theta \f(\X_s) ds  \in A\Big) = \oP_{x_0}(\theta  \in A).
		\end{align*}
	\end{enumerate}	
	Existence and uniqueness hold simultaneously for the MPs \((\a, \b, x_0)\) and \((\f^{-1}\a, \f^{-1} \b, x_0)\).
\end{theoremo}
Let \(x_0 \in \mathbb{R}^d\) and let \(\xp\) be a solution to the MP \((\a, \b, x_0)\). To simplify the notation, we denote \(P \equiv \xp\).
We start the proof by defining \(Y\). For \(t \in \mathbb{R}_+\) we set
\begin{align*}
T_t \triangleq \int_0^{t \wedge \theta} \f(\X_s)  ds,\qquad 
L_t \triangleq \inf ( s \in \mathbb{R}_+ \colon T_{s } > t ).
\end{align*}
The functions \(T, L\colon \mathbb{R}_+ \to [0, \infty]\) are increasing. Because \(\f\) is locally  bounded, we have \(T_{\theta_n \wedge n} < \infty\) for all \(n \in \mathbb{N}\). Using this and the strict positivity of \(\f\), we see that \(T\) is finite, absolutely continuous and strictly increasing on \([0, \theta)\). Moreover, because \(\lim_{t \nearrow \theta} T_t = T_\theta\) by the monotone convergence theorem, \(T\) is everywhere continuous.
We also note that \(L\) is finite, continuous and strictly increasing on \([0, T_{\theta})\) and everywhere right-continuous, and
\(L_{T_s} = s\) for \(s < \theta\) and \(T_{L_t} = t\) for \(t < T_{\theta}\), see \cite[pp. 7 -- 9]{RY}. In particular, we have
\[
\lim_{t \nearrow T_\theta} L_t = \lim_{t \nearrow \theta} L_{T_t} = \theta,
\]
and \(L\) is continuous on \(\{\theta = \infty\}\). 
For \(t \in \mathbb{R}_+\) we define 
\[
\Y_t \triangleq \begin{cases}\X_{L_t},& t < T_\theta,\\
\Delta,& t \geq T_\theta.\end{cases}
\]
It is easy to see that \(Y\) is right-continuous and measurable. 
	Because \(\{t < T_\theta\} = \{L_t < \theta\}\), for every \(t < T_\theta\) we have \(Y_t \in \mathbb{R}^d\) and consequently, \(T_\theta \leq \theta(Y)\).
	Noticing that  \(\theta (Y) \leq T_\theta\) by definition, we obtain that 
	\begin{align}\label{coro: exp time}
	T_\theta = \theta(Y) = \inf(t \in \mathbb{R}_+ \colon Y_t = \Delta).
	\end{align}
The following lemma shows that \(Y\) has almost surely continuous paths, which is part (i) in Theorem \ref{theo: tc}.
\begin{lemma}\label{lem: cont in T}
\(P\)-a.s. \(Y_{T_\theta -} = \Delta\) on \(\{T_\theta < \infty\}\).
\end{lemma}
\begin{discussion}
	On \(\{T_\theta < \infty, \theta < \infty\}\) we simply have
	\(
	Y_{T_\theta - } = X_\theta = \Delta,
	\)
	but on \(\{T_\theta < \infty, \theta = \infty\}\) it is necessary to understand the behavior of \(X_t\) as \(t \to \infty\). 
	We stress that \(\theta = \infty\) does not exclude \(T_\theta < \infty\) in a pathwise sense. To see this, consider the following simple example:
\[
\f (x) = \1_{(- \infty, 0)} (x) + \sum_{k = 1}^\infty a_k \1_{[k - 1, k)} (x), \qquad  x \in \mathbb{R}, 0 < a_k \leq 1.
\] 
Clearly, \(\f\) is locally bounded away from zero and infinity and for \(\omega (t) = t\) the integral
\[
\int_0^\infty \f (X_s (\omega)) ds  = \sum_{k = 1}^\infty a_k 
\]
converges or diverges depending on whether \((a_k)_{k \in \mathbb{N}}\) is summable or not.
To understand why Lemma \ref{lem: cont in T} holds, note that problems with the limit of \(X_t\) as \(t \to \infty\) occur for paths which either stay in a bounded subset of \(\mathbb{R}^d\) or have a recurrent behavior, where we think for instance of a one-dimensional Brownian path. These cases are excluded by considering the set \(\{T_\theta < \infty\}\), because for some bounded set \(U \subset \mathbb{R}^d\) the positive value \(\inf_{x \in \overline{U}} \f(x)\) will contribute to \(T_\theta\) for an infinite time.
The proof below borrows ideas from \cite[Lemma IV.2.1]{IW89}.
\end{discussion}
\begin{proof}
	For simplicity assume that \(\|x_0\| \leq 1\).
	For every \(n, m \in \mathbb{N}\) we define
	\begin{alignat*}{3}
	\sigma^m_1 &\triangleq 0, \qquad &&\tau^m_1 &&\triangleq \inf(t \in \mathbb{R}_+ \colon \|X_t\| \geq m + 1),\\
	\sigma^m_n&\triangleq \inf(t > \tau^m_{n - 1} \colon \|X_t\| \leq m),\qquad &&\tau^m_n &&\triangleq \inf(t > \sigma^m_n \colon \|X_t \| \geq m + 1).
	\end{alignat*}
	Set 
	\[
	\S \triangleq  \bigcap_{m \in \mathbb{N}} \bigcup_{n \in \mathbb{N}} \{\tau^m_n < \infty, \sigma^m_{n + 1} = \infty\},
	\]
	and note that \(\S \subseteq \{Y_{T_\theta - } = \Delta\}\). We show that \(P\)-a.s. \(\{T_\theta < \infty\} \subseteq \S\). More precisely, we show the equivalent inclusion \(P\)-a.s. \(\S^c \subseteq \{T_\theta = \infty\}\).
	We obtain
	\begin{align*}
	\S^c &=  \bigcup_{m \in \mathbb{N}} \bigcap_{n \in \mathbb{N}} \big(\{ \tau^m_n = \infty \} \cup \{ \sigma^m_{n + 1} < \infty\} \big)
	\\&= \bigcup_{m \in \mathbb{N}} \bigcap_{n \in \mathbb{N}} \big( \{ \tau^m_n = \infty, \sigma^m_n < \infty \} \cup \{ \sigma^m_{n} = \infty\} \cup \{\sigma^m_{n + 1} < \infty\} \big)
	\\& \subseteq \bigcup_{m \in \mathbb{N}} \Big( \Big(\bigcup_{k \in \mathbb{N}} \{ \tau^m_k = \infty, \sigma^m_k < \infty \}\Big) \cup \Big(  \bigcap_{n \in \mathbb{N}} \big( \{ \sigma^m_{n + 1} < \infty\} \cup \{\sigma^m_n = \infty\} \big) \Big) \Big)
		\\& = \bigcup_{m \in \mathbb{N}} \Big( \Big(\bigcup_{k \in \mathbb{N}} \{ \tau^m_k = \infty, \sigma^m_k < \infty \}\Big) \cup \Big( \bigcap_{n \in \mathbb{N}} \{\sigma^m_n < \infty\} \Big)\Big)
	\\&\subseteq \Big(\bigcup_{m \in \mathbb{N}} \bigcup_{k \in \mathbb{N}} \{ \tau^m_k = \infty, \sigma^m_k < \infty \} \Big) \cup \Big( \bigcup_{i \in \mathbb{N}} \bigcap_{n \in \mathbb{N}} \{\sigma^i_n < \infty\}\Big)
	\\&\triangleq \S_1 \cup \S_2.
	\end{align*}
	Take \(\omega \in \S_1\). Then, there exist \(n = n(\omega), m = m(\omega) \in \mathbb{N}\) such that \(\sigma^m_n (\omega) < \infty\) and \(\|X_t (\omega)\| \leq m + 1\) for all \(t \geq \sigma^m_n (\omega)\). Consequently, \(\theta (\omega) = \infty\) and 
	\[
	T_{\theta(\omega)} (\omega) = \int_0^\infty \f (X_s (\omega)) ds \geq \int_{\sigma_n (\omega)}^\infty \f(X_s (\omega)) ds \geq \inf_{\|y\| \leq m + 1} \f (y) \int_{\sigma_{n} (\omega)}^\infty ds = \infty.
	\]
	This implies \(\S_1 \subseteq \{T_\theta = \infty\}\).
	
	Set 
	\[
	\Theta \triangleq \bigcup_{m \in \mathbb{N}} \Big\{ \sigma_n^m < \infty \text{ for all } n \in \mathbb{N} \text{ and } \sum_{k = 1}^\infty \big(\tau^m_k - \sigma^m_k\big) = \infty \Big\}.
	\]
	Take \(\omega \in \Theta\) and let \(m = m(\omega) \in \mathbb{N}\) be as in the definition of \(\Theta\). Then, 
	\[
	T_{\theta(\omega)} (\omega) \geq \sum_{k = 1}^\infty \int_{\sigma^m_k (\omega)}^{\tau^m_k (\omega)} \f (\X_s (\omega)) ds \geq \inf_{\|y\| \leq m + 1} \f(y) \sum_{k =1}^\infty (\tau^m_k (\omega) - \sigma^m_k (\omega)) = \infty.
	\]
	This implies that \(\Theta \subseteq \{T_\theta = \infty\}\).
	
	Next, we show that \(P\)-a.s. \(\S_2 = \Theta\), which then implies that \(P\)-a.s. \(\S^c \subseteq \{T_\theta = \infty\}\) and thereby completes the proof. 
	We fix \(m, n \in \mathbb{N}\).
Clearly, we have on \(\{\sigma^m_n < \infty\}\)
	\begin{align*}
	\tau^m_n - \sigma^m_n &= \inf (t \in \mathbb{R}_+ \colon \|X_{t + \sigma^m_n}\| \geq m + 1) \triangleq \gamma.
	\end{align*}
	We set 
	\begin{align*}
	K &\triangleq \|X\|^2 - \|X_0\|^2 - \int_0^\cdot \big(2 \langle X_s, \b(X_s) \rangle + \textup{tr} (\a (X_s)\big) ds,
	\end{align*}
	and on \(\{\sigma^m_n < \infty\}\) we further set
	\begin{align*}
		M &\triangleq K_{\cdot \wedge \gamma + \sigma^m_n} - K_{\sigma^m_n},\\
		 I &\triangleq \int_{\sigma^m_n}^{\cdot \wedge \gamma + \sigma^m_n} \big(2 \langle X_s, \b(X_s) \rangle + \textup{tr} (\a (X_s))\big) ds.
	\end{align*}
	Using that for every \(t \in \mathbb{R}_+\) on \(\{\sigma^m_n< \infty\}\)
	\[
	\{\|X_{t \wedge \gamma + \sigma^m_n} \| \geq m + 1\} \subseteq \{|\|X_{t \wedge \gamma + \sigma^m_n}\|^2 - \|X_{\sigma^m_n}\|^2| \geq 1\} \subseteq \{|M_t| \geq \tfrac{1}{2}\} \cup \{|I_t| \geq \tfrac{1}{2}\},
	\]
	we obtain that 
	\begin{align*}
	\gamma &\geq \inf (t \in \mathbb{R}_+ \colon |M_t| \geq \tfrac{1}{2}) \wedge \inf (t \in \mathbb{R}_+ \colon |I_t| \geq \tfrac{1}{2}) \text{ on}\ \{\sigma^m_n < \infty\}.
	\end{align*}
	Because 
	\[
	|I_t|  \leq \sup_{\|y\| \leq m + 1} \big| 2\langle y, \b(y) \rangle + \textup{tr} (\a(y))\big| \ t \triangleq \alpha t  \text{ on}\ \{\sigma^m_n < \infty\},
	\]
	we obtain that
	\begin{align}\label{eq: good ineq 1}
	\inf (t \in \mathbb{R}_+ \colon |I_t| \geq \tfrac{1}{2}) \geq \frac{1}{2 \alpha} \text{ on}\ \{\sigma^m_n < \infty\}. 
	\end{align}
	
For every \(t \in \mathbb{R}_+\) we have \(t \wedge \gamma + \sigma^m_n < \theta\) on \(\{\sigma^m_n < \infty\}\). Consequently,  \begin{align}\label{eq: conv ST}
(t \wedge \gamma + \sigma^m_n) \wedge \theta_k \wedge k \nearrow t \wedge \gamma + \sigma^m_n \text{ as }k \to \infty\ \textup{on}\ \{\sigma^m_n < \infty\}. \end{align} 
Applying the definition of the martingale problem with \(f(x) = \|x\|^2\) yields that for every \(k \in \mathbb{N}\) the process \(K_{\cdot \wedge \theta_k \wedge k}\) is a \(P\)-martingale.
Note that for every \(t \in \mathbb{R}_+\)
\begin{align}\label{eq: DOM}
\sup_{k \in \mathbb{N}} |K_{(t \wedge \gamma + \sigma^m_n) \wedge \theta_k \wedge k} - K_{\sigma^m_n \wedge \theta_k \wedge k}| \1_{\{\sigma^m_n < \infty\}} \leq 2 (n + 1)^2 + \alpha t.
\end{align}
It is well-known that \(\sigma^m_n\) and \(\tau^m_n\) are \((\mathcal{F}^o_t)_{t \geq 0}\)-stopping times, see \cite[Proposition 2.1.5]{EK}. We note that \(t \wedge \gamma + \sigma^m_n\), which is set to be \(\infty\) in case \(\sigma^m_n = \infty\),  is an \((\mathcal{F}^o_t)_{t \geq 0}\)-stopping time, too. To see this, note that for all \(s \in \mathbb{R}_+\) 
\begin{align*}
\{t \wedge \gamma + \sigma^m_n \leq s \} = \{t + \sigma^m_n \leq s&, \sigma^m_n < \infty, t + \sigma^m_n \leq \tau^m_n\}\\ &\cup \{\tau^m_n \leq s, \sigma^m_n < \infty, \tau^m_n \leq t + \sigma^m_n\} \in \mathcal{F}^o_s, 
\end{align*}
which follows because for any \((\mathcal{F}^o_t)_{t \geq 0}\)-stopping times \(\rho\) and \(\tau\) the following hold: \(\mathcal{F}^o_{\rho} \cap \{\rho \leq s\} \subseteq \mathcal{F}^o_s, \{\rho \leq \tau\} \in \mathcal{F}^o_\rho \cap \mathcal{F}^o_\tau,\) and \(\mathcal{F}^o_\rho \subseteq \mathcal{F}^o_\tau\) whenever \(\rho \leq \tau\).

Let \(s < t\) and take \(A \in \mathcal{F}^o_{s + \sigma^m_n}\) and \(G \in \mathcal{F}^o_{\sigma^m_n}\). 
Recalling \eqref{eq: conv ST} and \eqref{eq: DOM}, the dominated convergence and the optional stopping theorem yield that 
\begin{align*}
E^P \big[ M_t &\1_A \1_{G} \1_{\{\sigma^m_n < \infty\}} \big] 
\\&= \lim_{k \to \infty} E^P \big[ \big(K_{(t \wedge \gamma + \sigma^m_n) \wedge \theta_k \wedge k} - K_{\sigma^m_n \wedge \theta_k \wedge k}\big) \1_A \1_{G} \1_{\{\sigma^m_n < \infty\}} \big]
\\&=  \lim_{k \to \infty} E^P \big[ \big(E^P \big[K_{(t \wedge \gamma + \sigma^m_n) \wedge \theta_k \wedge k}| \mathcal{F}_{s + \sigma^m_n}\big] - K_{\sigma^m_n \wedge \theta_k \wedge k}\big)\1_A \1_{G} \1_{\{\sigma^m_n < \infty\}} \big]
\\&=  \lim_{k \to \infty} E^P \big[ \big(K_{(s \wedge \gamma + \sigma^m_n) \wedge \theta_k \wedge k} - K_{\sigma^m_n \wedge \theta_k \wedge k}\big)\1_A \1_{G} \1_{\{\sigma^m_n < \infty\}} \big]
\\&= E^P \big[ M_s \1_A \1_G \1_{\{\sigma^m_n < \infty\}}\big].
\end{align*}
We conclude that there exists a \(P\)-null set \(N (s, t, A)\) such that 
\[
E^P \big[ \big(M_t - M_s \big) \1_A \1_{\{\sigma^m_n < \infty\}} | \mathcal{F}^o_{\sigma^m_n}\big] (\omega) = 0
\]
for all \(\omega \not \in N (s, t, A)\). Recall that \(\mathcal{F}^o_{s + \sigma^m_n} = \sigma (X_{t \wedge (s + \sigma^m_n)}, t \in \mathbb{Q}_+)\) is countably generated, see \cite[Theorem I.6]{aries2007optimal}, and let \(\mathcal{C}_s\) be a countable system of generators of \(\mathcal{F}^o_{s + \sigma^m_n}\). Set
\[
N \triangleq \bigcup_{t \in \mathbb{Q}_+} \bigcup_{\mathbb{Q}_+ \ni s < t} \bigcup_{A\in \mathcal{C}_s} N(s, t, A), 
\]
which is a \(P\)-null set. Now, we conclude that for all \(\omega \not \in N \cup \{\sigma^m_n = \infty\}\) the process \(M\) is a continuous \(P(\cdot | \mathcal{F}^o_{\sigma^m_n})(\omega)\)-martingale for the shifted filtration \((\mathcal{F}^o_{t + \sigma^m_n})_{t \geq 0}\) and, by the backwards martingale convergence theorem, also for its right-continuous version \(\F_{\sigma^m_n} \triangleq (\mathcal{F}_{t + \sigma^m_n})_{t \geq 0}\), see also \cite[Lemma 6.2]{Kallenberg}. 

Fix \(\omega \not \in N \cup \{\sigma^m_n = \infty\}\). It follows similar to \cite[Proposition VIII.3.3]{RY} that \(P(\cdot | \mathcal{F}^o_{\sigma^m_n})(\omega)\)-a.s.
\[
[M, M] = 4 \int_{\sigma^m_n}^{\cdot \wedge \gamma + \sigma^m_n} \langle X_s, \a (X_s) X_s\rangle ds.
\]
	The Dambis, Dubins--Schwarz theorem (see e.g., \cite[Theorem 16.4]{Kallenberg}) yields that on a standard extension of the filtered probability space \((\Omega, \mathcal{F}, \F_{\sigma^m_n}, P(\cdot | \mathcal{F}^o_{\sigma^m_n})(\omega))\), which we ignore in our notation for simplicity, there exists a one-dimensional Brownian motion \(B\) such that \(P(\cdot | \mathcal{F}^o_{\sigma^m_n})(\omega)\)-a.s.
	\(
	M = B_{[M, M]}.
	\)
	Because \(P(\cdot | \mathcal{F}^o_{\sigma^m_n})(\omega)\)-a.s.
	\[
	4 \int_{\sigma^m_n}^{t \wedge \gamma + \sigma^m_n} \langle X_s, \a (X_s) X_s\rangle ds \leq 4\Big( \sup_{\|y\| \leq m + 1} \langle y, \a (y) y\rangle \vee 1 \Big) t \triangleq \beta t,	\quad t \in \mathbb{R}_+,
	\]
	we have \(P(\cdot | \mathcal{F}^o_{\sigma^m_n})(\omega)\)-a.s.
	\begin{align}\label{eq: good ineq 2}
	\inf (t \in \mathbb{R}_+ \colon |B_{[M, M]_t}| \geq \tfrac{1}{2}) \geq \frac{\inf(t \in \mathbb{R}_+ \colon |B_t| \geq \tfrac{1}{2})}{\beta} \triangleq \frac{\tau}{\beta}. 
	\end{align}
	In summary, \eqref{eq: good ineq 1} and \eqref{eq: good ineq 2} imply that 
	\[
	E^P \big[ e^{- (\tau^m_n - \sigma^m_n)} | \mathcal{F}^o_{\sigma^m_n} \big] (\omega) \leq E \big[ e^{-  \frac{\tau}{\beta} \wedge \frac{1}{2 \alpha}}\big] \triangleq \C.
	\]
	We note that the law of \(\tau\) under \(P(\cdot | \mathcal{F}^o_{\sigma^m_n})(\omega)\) only depends on the Wiener measure, which means that \(\C\) is a constant independent of \(n, m\) and \(\omega\).
	Note also that \(\C < 1\). 

	Now, we obtain for all \(n \in \mathbb{Z}_+\)
	\begin{align*}
	E^P \Big[ \prod_{k = 1}^{n + 1}\hspace{0.05cm} &\1_{\{\sigma^m_{k} < \infty\}} e^{- (\tau^m_k - \sigma^m_k)}\Big] 
	\\&= E^P \Big[ \prod_{k = 1}^{n} \1_{\{\sigma^m_k < \infty\}}  e^{- (\tau^m_k - \sigma^m_k)} \1_{\{\sigma^m_{n + 1} < \infty\}} E^P \big[ e^{- (\tau^m_{n + 1} - \sigma^m_{n + 1})} \big| \mathcal{F}^o_{\sigma^m_{n + 1}}\big]\Big] 
	\\&\leq \C E^P \Big[ \prod_{k = 1}^{n} \1_{\{\sigma^m_k < \infty\}}  e^{- (\tau^m_k - \sigma^m_k)} \Big].
	\end{align*}
	By induction, we conclude
	\[
	E^P \Big[ \prod_{k = 1}^{n} \1_{\{\sigma^m_{k} < \infty\}} e^{- (\tau^m_k - \sigma^m_k)}\Big]  \leq \C^{n}, \quad n \in \mathbb{N}.
	\]
	Letting \(n \to \infty\) and using the dominated convergence theorem yields that
	\[
	E^P \Big[ \prod_{k = 1}^{\infty} \1_{\{\sigma^m_{k} < \infty\}} e^{- (\tau^m_k - \sigma^m_k)}\Big] = 0.
	\]
	This implies that \(P\)-a.s. for all \(m \in \mathbb{N}\)
	\[
	\prod_{k = 1}^{\infty} \1_{\{\sigma^m_{k} < \infty\}} e^{- (\tau^m_k - \sigma^m_k)} = \prod_{k = 1}^{\infty} \1_{\{\sigma^m_{k} < \infty\}} e^{- \sum_{i = 1}^\infty (\tau^m_i - \sigma^m_i)} = 0.
	\]
We conclude that \(P\)-a.s. \(\Sigma^c_2 \subseteq \Theta\). The proof is complete.
\end{proof}
\begin{remark}
	In case the MP \((\a, \b, x)\) has a unique solution \(P_x\) for all \(x \in \mathbb{R}^d\) and \(x \mapsto P_x\) is continuous, the proof of \(P\)-a.s. \(\mathcal{O}_2 = \Theta\) in Lemma \ref{lem: cont in T} simplifies substantially:
	It follows as in \cite[Lemma 11.1.2]{SV} that the map \(\omega \mapsto e^{- \tau_1^m (\omega)}\) is \(P_x\)-a.s. continuous for every \(x \in \mathbb{R}^d\). Thus, by the continuous mapping theorem, 
	\(x \mapsto E_x [e^{- \tau^m_1}]\) is continuous. Consequently, \(\C \triangleq \sup_{\|x\| \leq m} E_x [e^{- \tau^m_1}] < 1\). Now, using the strong Markov property, we obtain 
\begin{align*}
	E_{x_0} \Big[ \prod_{k = 1}^{n + 1}\hspace{0.05cm} &\1_{\{\sigma^m_{k} < \infty\}} e^{- (\tau^m_k - \sigma^m_k)}\Big] 
\\&= E_{x_0} \Big[ \prod_{k = 1}^{n} \1_{\{\sigma^m_k < \infty\}}  e^{- (\tau^m_k - \sigma^m_k)} \1_{\{\sigma^m_{n + 1} < \infty\}} E_{x_0} \big[ e^{- (\tau^m_{n + 1} - \sigma^m_{n + 1})} \big| \mathcal{F}^o_{\sigma^m_{n + 1}}\big]\Big] 
\\&= E_{x_0} \Big[ \prod_{k = 1}^{n} \1_{\{\sigma^m_k < \infty\}}  e^{- (\tau^m_k - \sigma^m_k)} \1_{\{\sigma^m_{n + 1} < \infty\}} E_{X_{\sigma^m_{n + 1}}} \big[ e^{- \tau^m_1}\big]\Big] 
\\&\leq \C E_{x_0} \Big[ \prod_{k = 1}^{n} \1_{\{\sigma^m_k < \infty\}}  e^{- (\tau^m_k - \sigma^m_k)} \Big]
\leq \C^{n + 1} \to 0 \text{ as } n \to \infty.
\end{align*}
The other proof of Lemma \ref{lem: cont in T} requires no uniqueness assumption on \(P\) and no continuity assumptions on \(\b\) and/or \(\a\), which are often required for \(x \mapsto P_x\) to be continuous, see \cite{pinsky1995positive,SV}.
\end{remark}

For \(n \in \mathbb{N}\) set 
\(
\gamma_n \triangleq T_{\theta_n \wedge n}
\)
and note that 
\(
L_{\gamma_n} = \theta_n \wedge n.
\)
It follows from \cite[Proposition V.1.4]{RY} that for all \(t \in \mathbb{R}_+\) and \(n \in \mathbb{N}\)
\[
L_{t \wedge \gamma_n} = \int_0^{L_{t \wedge \gamma_n}} \f^{-1} (\X_s) d T_s = \int_0^{t \wedge \gamma_n} \f^{-1}(X_{L_s}) d T_{L_s} = \int_0^{t \wedge \gamma_n} \f^{-1}(X_{L_s}) d s.
		\]
In other words, we have for all \(n \in \mathbb{N}\)
\begin{align}\label{eq: meas}
\1_{\{t \leq \gamma_n\}} d L_t = \1_{\{t \leq \gamma_n\}} \f^{-1}(\X_{L_t}) dt.
\end{align}
Using  \eqref{eq: meas} and again \cite[Proposition V.1.4]{RY}, we obtain for every locally bounded Borel function \(\g \colon \mathbb{R}^d \to \mathbb{R}\) that for all \(t \in \mathbb{R}_+\) and \(n \in \mathbb{N}\)
\begin{equation}\label{eq: tc}\begin{split}
\int_0^{t \wedge \gamma_n} \frac{\g(Y_s) ds}{\f (Y_s)} &= \int_0^{t \wedge \gamma_n} \frac{\g(\X_{L_s})ds}{\f(\X_{L_s})}
\\&= \int_0^{t \wedge \gamma_n} \g (\X_{L_s}) dL_s
\\&= \int_0^{L_{t \wedge \gamma_n}} \g (\X_s) ds.
\end{split}
\end{equation}

Note that \(L_t\) is an \(\F\)-stopping time.
Define the time-changed filtration \(\G = (\mathcal{G}_t)_{t \geq 0}  \triangleq (\mathcal{F}_{L_t})_{t \geq 0}\). Because \((L_t)_{t \geq 0}\) is right-continuous, also \(\G\) is right-continuous, and because \(\theta_n \wedge n\) is an \(\F\)-stopping time, \cite[Lemma 10.5]{J79} implies that \(\gamma_n = T_{\theta_n \wedge n}\) is a \(\G\)-stopping time and that \(t \mapsto L_{t \wedge \gamma_n}\) is an increasing sequence of \(\F\)-stopping times.

We set 
\[
\mathfrak{K} f \triangleq \langle \nabla f, \b \rangle + \tfrac{1}{2} \textup{tr} (\nabla^2 f \a), \quad f \in C^2(\mathbb{R}^d).
\]
Recall that, by the definition of the MP \((\a, \b, x_0)\), the process 
\[
f(X_{\cdot \wedge \theta_n}) - f(x_0) - \int_0^{\cdot \wedge \theta_n} \mathfrak{K} f (X_s) ds
\]
is a \(P\)-martingale. 

Recall further that \(L_{t \wedge \gamma_n} \leq \theta_n \wedge n\).
Using \eqref{eq: tc} and the optional stopping theorem, for \(s < t, n \in \mathbb{N}\) and \(f \in C^2 (\mathbb{R}^d)\) we obtain that \(P\)-a.s.
\begin{align*}
E^P \Big[ f(Y_{t \wedge \gamma_n}) &- f(x_0) - \int_0^{t \wedge \gamma_n} \frac{\mathfrak{K} f (Y_r) dr}{\f (Y_s)} \big| \mathcal{G}_s\Big]
\\&= 
E^P \Big[ f(\X_{L_{t \wedge \gamma_n} \wedge \theta_n \wedge n}) - f(x_0) - \int_0^{L_{t \wedge \gamma_n} \wedge \theta_n \wedge n} \mathfrak{K} f (\X_r) dr \big| \mathcal{F}_{L_s} \Big] 
\\&= 
f(\X_{L_{t \wedge \gamma_n} \wedge \theta_n \wedge n \wedge L_s}) - f(x_0) - \int_0^{L_{t \wedge \gamma_n} \wedge \theta_n \wedge n \wedge L_s} \mathfrak{K} f (\X_r) dr
\\&= f(\X_{L_{s \wedge \gamma_n}}) - f(x_0) - \int_0^{L_{s \wedge \gamma_n}} \mathfrak{K} f (\X_r) dr 
\\&= f(Y_{s \wedge \gamma_n}) - f(x_0) - \int_0^{s \wedge \gamma_n} \frac{\mathfrak{K} f(Y_r) dr}{\f (Y_r)}.
\end{align*}
This yields that 
\[
f(Y_{\cdot \wedge \gamma_n}) - f(x_0) - \int_0^{\cdot \wedge \gamma_n} \frac{\mathfrak{K} f (Y_r) dr}{\f (Y_s)}
\]
is a \(P\)-martingale for the filtration \(\G\). In particular, by  \cite[Lemma II.67.10]{RW1}, it is a \(P\)-martingale for the \(P\)-augmentation of \(\G\), which we denote by \(\G^P\).
We redefine the process \(Y\) on a \(P\)-null set such that it gets everywhere continuous paths. With abuse of notation we denote the redefined process still by \(Y\). Note that \(\theta_n (Y)\) is a \(\G^P\)-stopping time.
Recalling that \(P\)-a.s. \(\gamma_n \nearrow T_\theta = \theta (Y)\) and \(t \wedge \theta_n (Y) < \theta (Y)\), the dominated convergence theorem yields for all \(s < t, n \in \mathbb{N}\) and \(f \in C^2 (\mathbb{R}^d)\) that \(P\)-a.s.
\begin{align*}
E^P \Big[ f(Y_{t \wedge \theta_n (Y)}) &- f(x_0) - \int_0^{t \wedge \theta_n (Y)} \frac{\mathfrak{K} f (Y_r) dr}{\f (Y_s)} \big| \mathcal{G}^P_s\Big]
\\&= \lim_{m \to \infty} E^P \Big[ f(Y_{t \wedge \theta_n (Y) \wedge \gamma_m}) - f(x_0) - \int_0^{t \wedge \theta_n (Y) \wedge \gamma_m} \frac{\mathfrak{K} f (Y_r) dr}{\f (Y_s)} \big| \mathcal{G}^P_s\Big]
\\&= \lim_{m \to \infty} \Big( f(Y_{s \wedge \theta_n (Y) \wedge \gamma_m}) - f(x_0) - \int_0^{s \wedge \theta_n(Y) \wedge \gamma_m} \frac{\mathfrak{K} f(Y_r) dr}{\f (Y^n_r)} \Big)
\\&= f(Y_{s \wedge \theta_n (Y)}) - f(x_0) - \int_0^{s \wedge \theta_n(Y)} \frac{\mathfrak{K} f(Y_r) dr}{\f (Y^n_r)}.
\end{align*}
Using the tower rule, we conclude that 
\[
f(Y_{\cdot \wedge \theta_n (Y)}) - f(x_0) - \int_0^{\cdot \wedge \theta_n (Y)} \frac{\mathfrak{K} f (Y_r) dr}{\f (Y_s)}
\]
is a \(P\)-martingale for the filtration generated by \(Y\). Consequently, the push-forward \(P \circ \Y^{-1}\) solves the MP \((\f^{-1}\a, \f^{-1}\b, x_0)\), which is part (iii) of Theorem \ref{theo: tc}. Recalling \eqref{coro: exp time} shows the formula \eqref{eq: exp id}, i.e. part (iv) of Theorem \ref{theo: tc}.

We now introduce the process \(U\) and verify (ii) in Theorem \ref{theo: tc}.
For \(t \in \mathbb{R}_+\) we define
\begin{align*}
S_t \triangleq \int_0^{t \wedge \theta} \f^{-1}(\X_s)  ds,\qquad 
A_t \triangleq \inf ( s \in \mathbb{R}_+ \colon S_{s} > t ),
\end{align*}
and 
\[
U_t \triangleq \begin{cases} X_{A_t}, & t < S_\theta,\\
\Delta,& t \geq S_\theta.
\end{cases}
\]
Again, it is easy to see that \(U\) is right-continuous and measurable. Using \eqref{coro: exp time} and \eqref{eq: meas}, we obtain \(P\)-a.s. for all \(t \in \mathbb{R}_+\)
\[
S_{t} \circ Y = \int_0^{t \wedge T_\theta} \f^{-1} (Y_s) ds = \lim_{n \to \infty} \int_0^{t \wedge \gamma_n \wedge n} \f^{-1} (X_{L_s}) ds = \lim_{n \to \infty} L_{t \wedge \gamma_n \wedge n} = L_t.
\]
In particular, \(P\)-a.s. \(S_\theta \circ Y = \theta\). Now, we obtain \(P\)-a.s. for \(t \in \mathbb{R}_+\)
\[
A_t \circ Y = \inf (s \in \mathbb{R}_+ \colon L_s > t) = T_t,
\]
which implies \(X_{A_t} \circ Y = X_{L_{T_t}} = X_t\) for all \(t < S_\theta \circ Y=  \theta\).
In summary, we conclude that \(P\)-a.s. \(U \circ Y = X\), i.e. that (ii) in Theorem \ref{theo: tc} holds.

Finally, let us explain that if the MP \((\f^{-1}\a, \f^{-1} \b, x_0)\) has at most one solution, then \(P\) is the unique solution to the MP \((\a, \b, x_0)\).
For \(n \in \mathbb{N}\) let \(0 \leq t_1 < t_2 < \dots < t_n < \infty\) and \(G_1, \dots, G_n \in \mathcal{B}(\mathbb{R}^d_\Delta)\). 
Suppose that \(Q\) is a second solution to the MP \((\a, \b, x_0)\). Then, the push-forwards \(P \circ Y^{-1}\) and \(Q\circ Y^{-1}\) both solve the MP \((\f^{-1}\a, \f^{-1}\b, x_0)\) and we deduce from the uniqueness assumption and (ii) in Theorem \ref{theo: tc} that
\begin{align*}
P(X_{t_1} \in G_1, \dots, X_{t_n} \in G_n) 
&= P \circ Y^{-1} (U_{t_1} \in G_1, \dots, U_{t_n} \in G_n) 
\\&= Q \circ Y^{-1} (U_{t_1} \in G_1, \dots, U_{t_n} \in G_n)
\\&= Q(X_{t_1} \in G_1, \dots, X_{t_n} \in G_n).
\end{align*}
By a monotone class argument, \(P = Q\). The proof is complete.
\qed

\appendix
\section{A Few Existence and Uniqueness Results}\label{app}
In this appendix we collect some existence and uniqueness results for martingale problems. 
We assume that \(\b \colon \mathbb{R}^d \to \mathbb{R}^d\) and \(\a \colon \mathbb{R}^d \to \mathbb{S}^d\) are locally bounded and we
formulate the following conditions:
\begin{enumerate}
	\item[\textup{(A1)}] \(\b\) and \(\a\) are continuous. 
	\item[\textup{(A2)}] \(\textup{det} (\a)\) is locally bounded away from zero.
	\item[\textup{(A3)}] \(\a\) is continuous and \(\langle \xi, \a (x) \xi \rangle > 0\) for all \(x \in \mathbb{R}^d\) and \(\xi \in \mathbb{R}^d \backslash \{0\}\). 
	\item[\textup{(A4)}] \(\b\) is locally Lipschitz continuous and \(\a\) has a locally Lipschitz continuous root.
\end{enumerate}
We use this opportunity and illustrate that Theorem \ref{theo: tc} can also be used to obtain existence and uniqueness criteria for martingale problems.
\begin{proposition}\label{prop: EU}
	Let \(x_0 \in \mathbb{R}^d\).
	If \textup{(A1)} or \textup{(A2)} holds, then there exists a solution \(\xp\) to the MP \((\a, \b, x_0)\).
	If \textup{(A3)} or \textup{(A4)} holds, then there exists a unique solution \(\xp\) to the MP \((\a, \b, x_0)\).
\end{proposition}
\begin{proof}
	Let \(\f \colon \mathbb{R}^d \to (0, \infty)\) be a continuous function such that \(\f \a\) and \(\f \b\) are bounded. Such a function can be constructed as follows:
	Set \[\g \triangleq \sum_{k = 1}^\infty a^{-1}_k \1_{[k - 1, k)}, \text{ where } a_k \triangleq \sup_{\|x\| \leq k} \|\b(x)\| \vee \sup_{\|x\| \leq k} \|\a(x)\| \vee 1,\] and let \(\z \colon \mathbb{R}_+ \to (0, \infty)\) be a continuous function \(\z \leq \g\). Then, \(\f (x) \triangleq \z (\|x\|)\) has the claimed properties.
	In case one of (A1) -- (A3) holds, the MP \((\f \a, \f \b, x_0)\) has a (conservative) solution and in case (A3) holds the solution is even unique. With these observations at hand, Theorem \ref{theo: tc} implies that existence holds for the MP \((\a, \b, x_0)\) under either of (A1) -- (A4) and that uniqueness holds under (A3). That uniqueness also holds under (A4) is well-known, see \cite[Theorem IV.3.1]{IW89}. 
	Finally, we provide references for the existence and uniqueness statements concerning the MP \((\f \a, \f \b, x_0)\): For existence under (A1) and (A2) see \cite[Theorem 6.1.7]{SV} and \cite[Corollary 1]{ROZKOSZ1991187} respectively, and for existence and uniqueness under (A3) see \cite[Theorem 7.2.1]{SV}.
\end{proof}
\begin{remark}
	Existence under (A1) is also implied by \cite[Theorems IV.2.3]{IW89} and existence and uniqueness under (A3) is implied by \cite[Theorems 1.13.1]{pinsky1995positive}. 
	Local integrability conditions, which are weaker than (A2), for the existence of a solution with not necessarily continuous paths (the left-limit at \(\theta\) needs not to exist) are given in \cite[Theorem 3.1]{10.2307/44239216}.
\end{remark}

%

\begin{thebibliography}{10}
	
	\bibitem{BENARI2005179}
	I. Ben-Ari and R. Pinsky.
	\newblock Absolute continuity/singularity and relative entropy properties for
	probability measures induced by diffusions on infinite time intervals.
	\newblock {\em Stochastic Processes and their Applications}, 115(2):179 -- 206,
	2005.
	
	\bibitem{BHATTACHARYA198295}
	R. Bhattacharya and S.~Ramasubramanian.
	\newblock Recurrence and ergodicity of diffusions.
	\newblock {\em Journal of Multivariate Analysis}, 12(1):95 -- 122, 1982.
	
	\bibitem{Salminen06}
	A. Borodin and P.~Salminen.
	\newblock Some exponential integral functionals of BM \((\mu)\) and Bes (3).
	\newblock {\em Journal of Mathematical Sciences}, 133(3):1231--1248, 2006.
	
	\bibitem{bottcher2014levy}
	B.~B{\"o}ttcher, R.~Schilling, and J.~Wang.
	\newblock {\em L{\'e}vy Matters III: L{\'e}vy-Type Processes: Construction,
		Approximation and Sample Path Properties}.
	\newblock Lecture Notes in Mathematics. Springer International Publishing,
	2014.
	
	\bibitem{CFY}
	P.~Cheridito, D.~Filipovic, and M.~Yor.
	\newblock Equivalent and absolutely continuous measure changes for
	jump-diffusion processes.
	\newblock {\em The Annals of Applied Probability}, 15(3):1713--1732, 2005.
	
	\bibitem{Cherny2006}
	A.~Cherny and M.~Urusov.
	\newblock On the absolute continuity and singularity of measures on filtered
	spaces: Separating times.
	\newblock In {\em From Stochastic Calculus to Mathematical Finance: The
		Shiryaev Festschrift}, pages 125--168. Springer Berlin Heidelberg, 2006.
	
	\bibitem{doi:10.1142/S0219024918500024}
	D. Criens.
	\newblock Deterministic criteria for the absence and existence of arbitrage in
	multi-dimensional diffusion markets.
	\newblock {\em International Journal of Theoretical and Applied Finance},
	21(01):1850002, 2018.
	
	\bibitem{criens20}
	D. Criens.
	\newblock No arbitrage in continuous financial markets.
	\newblock {\em To appear in Mathematics and Financial Economics}, 2020.
	
	\bibitem{criensglau2018}
	D. Criens and K. Glau.
	\newblock Absolute continuity of semimartingales.
	\newblock {\em Electronic Journal of Probability}, 23(125):1--28, 2018.
	
	\bibitem{CUI2014118}
	Z. Cui.
	\newblock A new proof of an Engelbert--Schmidt type zero–one law for
	time-homogeneous diffusions.
	\newblock {\em Statistics \& Probability Letters}, 89:118 -- 123, 2014.
	
	\bibitem{DS}
	F.~Delbaen and W.~Schachermayer.
	\newblock {\em The Mathematics of Arbitrage}.
	\newblock Springer, 2006.
	
	\bibitem{doi:10.1002/mana.19871310120}
	H.-J. Engelbert and W.~Schmidt.
	\newblock On the behaviour of certain Bessel functional. an application to a
	class of stochastic differential equations.
	\newblock {\em Mathematische Nachrichten}, 131(1):219--234, 1987.
	
	\bibitem{EK}
	S.~Ethier and T.~Kurtz.
	\newblock {\em Markov Processes: Characterization and Convergence}.
	\newblock Wiley, 2005.
	
	\bibitem{IW89}
	N.~Ikeda and S.~Watanabe.
	\newblock {\em Stochastic Differential Equations and Diffusion Processes}.
	\newblock Elsevier Science, 1989.
	
	\bibitem{J79}
	J.~Jacod.
	\newblock {\em Calcul stochastique et probl{\`e}mes de martingales}.
	\newblock Number 714 in Lecture notes in Mathematics. Springer, 1979.
	
	\bibitem{Kallenberg}
	O.~Kallenberg.
	\newblock {\em Foundations of Modern Probability}.
	\newblock Probability and Its Applications. Springer New York, 2006.
	
	\bibitem{KaraShre}
	I.~Karatzas and S.~Shreve.
	\newblock {\em Brownian Motion and Stochastic Calculus}.
	\newblock Springer, 2nd edition, 1991.
	
	\bibitem{khoshnevisan2006}
	D. Khoshnevisan, P. Salminen, and M. Yor.
	\newblock A note on a.s. finiteness of perpetual integral functionals of
	diffusions.
	\newblock {\em Electronic Communications in Probability}, 11:108--117, 2006.
	
	\bibitem{kuhn2019}
	F. K{\"u}hn.
	\newblock Perpetual integrals via random time changes.
	\newblock {\em Bernoulli}, 25(3):1755--1769, 2019.
	
	\bibitem{10.2307/44239216}
	V. Kurenok and A. Lepeyev.
	\newblock On multi-dimensional SDEs with locally integrable coefficients.
	\newblock {\em The Rocky Mountain Journal of Mathematics}, 38(1):139--174,
	2008.
	
	\bibitem{doi:10.1111/j.1467-9965.2012.00530.x}
	A. Lyasoff.
	\newblock The two fundamental theorems of asset pricing for a class of
	continuous-time financial markets.
	\newblock {\em Mathematical Finance}, 24(3):485--504, 2014.
	
	\bibitem{mckean1969stochastic}
	H. McKean.
	\newblock {\em Stochastic Integrals}.
	\newblock Academic Press, 1969.
	
	\bibitem{mijatovic2012}
	A.~Mijatovi{\'c} and M.~Urusov.
	\newblock Convergence of integral functionals of one-dimensional diffusions.
	\newblock {\em Electronic Communications in Probability}, 17(61):1--13, 2012.
	
	\bibitem{MU(2012)}
	A~Mijatovi{\'c} and M.~Urusov.
	\newblock On the martingale property of certain local martingales.
	\newblock {\em Probability Theory and Related Fields}, 152:1--30, 2012.
	
	\bibitem{MU-det}
	A. Mijatovi{\'c} and M.~Urusov.
	\newblock Deterministic criteria for the absence of arbitrage in
	one-dimensional diffusion models.
	\newblock {\em Finance and Stochastics}, 16(2):225--247, 2012.
	
	\bibitem{morters_peres_2010}
	P. M{\"o}rters and Y. Peres.
	\newblock {\em Brownian Motion}.
	\newblock Cambridge Series in Statistical and Probabilistic Mathematics.
	Cambridge University Press, 2010.
	
	\bibitem{MUSIELA198679}
	M.~Musiela.
	\newblock On Kac functionals of one-dimensional diffusions.
	\newblock {\em Stochastic Processes and their Applications}, 22(1):79 -- 88,
	1986.
	
	\bibitem{pinsky1995positive}
	R.~Pinsky.
	\newblock {\em Positive Harmonic Functions and Diffusion}.
	\newblock Cambridge University Press, 1995.
	
	\bibitem{RY}
	D.~Revuz and M.~Yor.
	\newblock {\em Continuous Martingales and Brownian Motion}.
	\newblock Springer, 3rd edition, 1999.
	
	\bibitem{RW1}
	L.~Rogers and D.~Williams.
	\newblock {\em Diffusions, Markov Processes, and Martingales: Volume 1,
		Foundations}.
	\newblock Cambridge University Press, 2000.
	
	\bibitem{ROZKOSZ1991187}
	A. Rozkosz and L. Słomiński.
	\newblock On existence and stability of weak solutions of multidimensional
	stochastic differential equations with measurable coefficients.
	\newblock {\em Stochastic Processes and their Applications}, 37(2):187 -- 197,
	1991.
	
	\bibitem{RufSDE}
	J.~Ruf.
	\newblock The martingale property in the context of stochastic differential
	equations.
	\newblock {\em Electronic Communications in Probability}, 20(34):1--10, 2015.
	
	\bibitem{aries2007optimal}
	A.~Shiryaev.
	\newblock {\em Optimal Stopping Rules}.
	\newblock Springer, 1978.
	
	\bibitem{Shir}
	A.~Shiryaev.
	\newblock {\em Essentials of Stochastic Finance: Facts, Models, Theory}.
	\newblock Advanced series on statistical science \& applied probability. World
	Scientific, 1999.
	
	\bibitem{Sin}
	C.~Sin.
	\newblock Complications with stochastic volatility models.
	\newblock {\em Advances in Applied Probability}, 30(1):256--268, 1998.
	
	\bibitem{SV}
	D.~Stroock and S.~Varadhan.
	\newblock {\em Multidimensional Diffussion Processes}.
	\newblock Springer, 2nd edition, 1997.
	
	\bibitem{doi:10.1137/1103025}
	V. Volkonskii.
	\newblock Random substitution of time in strong Markov processes.
	\newblock {\em Theory of Probability \& Its Applications}, 3(3):310--326, 1958.
	
	\bibitem{10.1007/BFb0083762}
	X.-X. Xue.
	\newblock A zero-one law for integral functionals of the Bessel process.
	\newblock In Jacques Az{\'e}ma, Marc Yor, and Paul~Andr{\'e} Meyer, editors,
	{\em S{\'e}minaire de Probabilit{\'e}s XXIV 1988/89}, pages 137--153, Berlin,
	Heidelberg, 1990. Springer Berlin Heidelberg.
	
\end{thebibliography}

\end{document}